\documentclass{amsart}
\usepackage{amssymb}
\usepackage{amsmath}
\usepackage{amsfonts}

\newtheorem{theorem}{Theorem}
\theoremstyle{plain}

\newtheorem{corollary}{Corollary}

\newtheorem{example}{Example}

\newtheorem{lemma}{Lemma}

\numberwithin{equation}{section}
\input{tcilatex}

\begin{document}
\title[]{Intra regular Abel-Grassmann's groupoids characterized by their
intuitionistic fuzzy ideals}
\author{}
\maketitle

\begin{center}
$^{\ast }$\textbf{Madad Khan and Faisal}

\textbf{Department of Mathematics}

\textbf{COMSATS Institute of Information Technology}

\textbf{Abbottabad, Pakistan.}

$^{\ast }$\textbf{E-mail: madadmath@yahoo.com}, \textbf{E-mail:
yousafzaimath@yahoo.com}
\end{center}

\QTP{Body Math}
$\bigskip $

\textbf{Abstract. }In this paper, we have discussed the properties of
intuitionistic fuzzy ideals of an AG-groupoids. We have characterized an
intra-regular AG-groupoid in terms of intuitionistic fuzzy left (right,
two-sided) ideals, fuzzy (generalized) bi-ideals, intuitionistic fuzzy
interior ideals and intuitionistic fuzzy quasi ideals. We have proved that
the intuitionistic fuzzy left (right, interior, quasi) ideal coincides in an
intra-regular AG-groupoid. We have also shown that the set of intuitionistic
fuzzy two-sided ideals of an intra-regular AG-groupoid forms a semilattice
structure.


\begin{center}
\bigskip

{\LARGE Introduction}
\end{center}

Given a set $S$, a fuzzy subset of $S$ is an arbitrary mapping $%
f:S\rightarrow \lbrack 0,1]$ where $[0,1]$ is the unit segment of a real
line. This fundamental concept of fuzzy set was first given by Zadeh \cite%
{L.A.Zadeh} in $1965$. Fuzzy groups have been first considered by Rosenfeld 
\cite{A. Rosenfeld} and fuzzy semigroups by Kuroki \cite{N. Kuroki}.

Atanassov \cite{at}, introduced the concept of an intuitionistic fuzzy set.
Dengfeng and Chunfian \cite{8} introduced the concept of the degree of
similarity between intuitionistic fuzzy sets, which may be finite or
continuous, and gave corresponding proofs of these similarity measure and
discussed applications of the similarity measures between intuitionistic
fuzzy sets to pattern recognition problems. Jun in \cite{10}, introduced the
concept of an intuitionistic fuzzy bi-ideal in ordered semigroups and
characterized the basic properties of ordered semigroups in terms of
intuitionistic fuzzy bi-ideals. In \cite{15} and \cite{16}, Kim and Jun
introduced the concept of intuitionistic fuzzy interior ideals of
semigroups. In \cite{20}, Shabir and Khan gave the concept of an
intuitionistic fuzzy interior ideal of ordered semigroups and characterized
different classes of ordered semigroups in terms of intuitionistic fuzzy
interior ideals. They also gave the concept of an intuitionistic fuzzy
generalized bi-ideal in \cite{21} and discussed different classes of ordered
semigroups in terms of intuitionistic fuzzy generalized bi-ideals.

In this paper, we consider the intuitionistic fuzzification of the concept
of several ideals in AG-groupoid and investigate some properties of such
ideals.

An AG-groupoid is a non-associative algebraic structure mid way between a
groupoid and a commutative semigroup. The left identity in an AG-groupoid if
exists is unique \cite{Mus3}. An AG-groupoid is non-associative and
non-commutative algebraic structure, nevertheless, it posses many
interesting properties which we usually find in associative and commutative
algebraic structures. An AG-groupoid with right identity becomes a
commutative monoid \cite{Mus3}. An AG-groupoid is basically the
generalization of semigroup (see \cite{Kaz}) with wide range of applications
in theory of flocks \cite{nas}. The theory of flocks tries to describes the
human behavior and interaction.

The concept of an Abel-Grassmann's groupoid (AG-groupoid) \cite{Kaz} was
first given by M. A. Kazim and M. Naseeruddin in $1972$ and they called it
left almost semigroup (LA-semigroup). P. Holgate call it simple invertive
groupoid \cite{hol}. An AG-groupoid is a groupoid having the left invertive
law%
\begin{equation}
(ab)c=(cb)a\text{, for\ all }a\text{, }b\text{, }c\in S\text{.}  \tag{$1$}
\end{equation}%
In an AG-groupoid, the medial law \cite{Kaz} holds%
\begin{equation}
(ab)(cd)=(ac)(bd)\text{, for\ all }a\text{, }b\text{, }c\text{, }d\in S\text{%
.}  \tag{$2$}
\end{equation}%
In an AG-groupoid $S$ with left identity, the paramedial law holds%
\begin{equation}
(ab)(cd)=(dc)(ba),\text{ for\ all }a,b,c,d\in S.  \tag{$3$}
\end{equation}%
If an AG-groupoid contains a left identity, the following law holds

\begin{equation}
a(bc)=b(ac)\text{, for\ all }a\text{, }b\text{, }c\in S\text{.}  \tag{$4$}
\end{equation}

\begin{center}
\bigskip

{\LARGE Preliminaries}
\end{center}

Let $S$ be an AG-groupoid, by an AG-subgroupoid of $S,$ we means a non-empty
subset $A$ of $S$ such that $A^{2}\subseteq A$.

A non-empty subset $A$ of an AG-groupoid $S$ is called a left (right) ideal
of $S$ if $SA\subseteq A$ $(AS\subseteq A)$.

A non-empty subset $A$ of an AG-groupoid $S$ is called a two-sided ideal or
simply ideal if it is both a left and a right ideal of $S$.

A non empty subset $A$ of an AG-groupoid $S$ is called a generalized
bi-ideal of $S$ if $(AS)A\subseteq A$.

An AG-subgroupoid $A$ of $S$ is called a bi-ideal of $S$ if $(AS)A\subseteq
A $.

A non empty subset $A$ of an AG-groupoid $S$ is called an interior ideal of $%
S$ if $(SA)S\subseteq A$.

A non empty subset $A$ of an AG-groupoid $S$ is called an quasi ideal of $S$
if $AS\cap SA\subseteq A$.

\bigskip

A fuzzy subset $f$ is a class of objects with a grades of membership having
the form

\begin{center}
$f=\{(x,$ $f(x))/x\in S\}.$
\end{center}

An intuitionistic fuzzy set (briefly, $IFS$) $A$ in a non empty set $S$ is
an object having the form

\begin{center}
$A=\left\{ (x,\mu _{A}(x),\gamma _{A}(x))/x\in S\right\} .$
\end{center}

The functions $\mu _{A}:S\longrightarrow \lbrack 0,1]$ and $\gamma
_{A}:S\longrightarrow \lbrack 0,1]$ denote the degree of membership and the
degree of nonmembership respectively such that for all $x\in S,$ we have

\begin{center}
$0\leq \mu _{A}(x)+\gamma _{A}(x)\leq 1.$
\end{center}

\QTP{Body Math}
For the sake of simplicity, we shall use the symbol $A=(\mu _{A},\gamma
_{A}) $ for an $IFS$ $A=\left\{ (x,\mu _{A}(x),\gamma _{A}(x))/x\in
S\right\} .$

\QTP{Body Math}
Let $\delta =\left\{ (x,S_{\delta }(x),\Theta _{\delta }(x))/S_{\delta }(x)=1%
\text{ and }\Theta _{\delta }(x)=0/x\in S\right\} =(S_{\delta },\Theta
_{\delta })$ be an $IFS,$ then $\delta =(S_{\delta },\Theta _{\delta })$
will be carried out in operations with an $IFS$ $A=(\mu _{A},\gamma _{A})$
such that $S_{\delta }$ and $\Theta _{\delta }$ will be used in
collaboration with $\mu _{A}$ and $\gamma _{A}$ respectively.

An $IFS$ $A=(\mu _{A},\gamma _{A})$ of an AG-groupoid $S$ is called an
intuitionistic fuzzy AG-subgroupoid of $S$ if $\mu _{A}(xy)\geq \mu
_{A}(x)\wedge \mu _{A}(y)$ and $\gamma _{A}(xy)\leq \gamma _{A}(x)\vee
\gamma _{A}(y)$ for all $x$, $y\in S.$

An $IFS$ $A=(\mu _{A},\gamma _{A})$ of an AG-groupoid $S$ is called an
intuitionistic fuzzy left ideal of $S$ if $\mu _{A}(xy)\geq \mu _{A}(y)$ and 
$\gamma _{A}(xy)\leq \gamma _{A}(y)$ for all $x$, $y\in S.$

An $IFS$ $A=(\mu _{A},\gamma _{A})$ of an AG-groupoid $S$ is called an
intuitionistic fuzzy right ideal of $S$ if $\mu _{A}(xy)\geq \mu _{A}(x)$
and $\gamma _{A}(xy)\leq \gamma _{A}(x)$ for all $x$, $y\in S.$

An $IFS$ $A=(\mu _{A},\gamma _{A})$ of an AG-groupoid $S$ is called an
intuitionistic fuzzy two-sided ideal of $S$ if it is both an intuitionistic
fuzzy left and an intuitionistic fuzzy right ideal of $S$.

An $IFS$ $A=(\mu _{A},\gamma _{A})$ of an AG-groupoid $S$ is called an
intuitionistic fuzzy generalized bi-ideal of $S$ if $\mu _{A}((xa)y)\geq \mu
_{A}(x)\wedge \mu _{A}(y)$ and $\gamma _{A}((xa)y)\leq \gamma _{A}(x)\vee
\gamma _{A}(y)$ for all $x$, $a$ and $y\in S$.

An intuitionistic fuzzy AG-subgroupoid $A=(\mu _{A},\gamma _{A})$ of an
AG-groupoid $S$ is called an intuitionistic fuzzy bi-ideal of $S$ if $\mu
_{A}((xa)y)\geq \mu _{A}(x)\wedge \mu _{A}(y)$ and $\gamma _{A}((xa)y)\leq
\gamma _{A}(x)\vee \gamma _{A}(y)$ for all $x$, $a$ and $y\in S$.

An $IFS$ $A=(\mu _{A},\gamma _{A})$ of an AG-groupoid $S$ is called an
intuitionistic fuzzy interior ideal of $S$ if $\mu _{A}((xa)y)\geq \mu
_{A}(a)$ and $\gamma _{A}((xa)y)\leq \gamma _{A}(a)$ for all $x$, $a$ and $%
y\in S$.

An $IFS$ $A=(\mu _{A},\gamma _{A})$ of an AG-groupoid $S$ is called an
intuitionistic fuzzy quasi ideal of $S$ if $(\mu _{A}\circ S)\cap (S\circ
\mu _{A})\subseteq \mu _{A}$ and $(\gamma _{A}\circ S)\cup (S\circ \gamma
_{A})\supseteq \gamma _{A},$ that is, $(A\circ \delta )\cap (\delta \circ
A)\subseteq A.$

Let $S$ be an AG-groupoid and let $A_{I}=\{A/$ $A\in S\},$ where $A=(\mu
_{A},\gamma _{A})$ be any $IFS$ of $S,$ then $(A_{I},\circ )$ satisfies $%
(1), $ $(2),$ $(3)$ and $(4)$.

An element $a$ of an AG-groupoid $S$ is called an intra-regular if there
exist $x,y\in S$ such that $a=(xa^{2})y$ and $S$ is called an intra-regular
if every element of $S$ is an intra-regular.

\begin{example}
\label{ex}Let $S=\{1,2,3,4,5\}$ be an AG-groupoid with left identity $4$
with the following multiplication table.
\end{example}

\begin{center}
\begin{tabular}{l|lllll}
. & $1$ & $2$ & $3$ & $4$ & $5$ \\ \hline
$1$ & $1$ & $1$ & $1$ & $1$ & $1$ \\ 
$2$ & $1$ & $2$ & $2$ & $2$ & $2$ \\ 
$3$ & $1$ & $2$ & $4$ & $5$ & $3$ \\ 
$4$ & $1$ & $2$ & $3$ & $4$ & $5$ \\ 
$5$ & $1$ & $2$ & $5$ & $3$ & $4$%
\end{tabular}
\end{center}

\QTP{Body Math}
It is easy to see that $S$ is an intra-regular. Define an $IFS$ $A=(\mu
_{A},\gamma _{A})$ of $S$ as follows: $\mu _{A}(1)=1$, $\mu _{A}(2)=$ $\mu
_{A}(3)=$ $\mu _{A}(4)=$ $\mu _{A}(5)=0,$ $\gamma _{A}(1)=0.3,$ $\gamma
_{A}(2)=0.4$ and $\gamma _{A}(3)=\gamma _{A}(4)=\gamma _{A}(5)=0.2,$ then
clearly $A=(\mu _{A},\gamma _{A})$ is an intuitionistic fuzzy two-sided
ideal and also an intuitionistic fuzzy AG-subgroupoid of $S$.

For an $IFS$ $A=(\mu _{A},\gamma _{A})$ of an AG-groupoid $S$ and $\alpha
\in (0,1],$ the set

\begin{center}
$A_{\alpha }=\{x\in S:\mu _{A}(x)\geq \alpha ,$ $\gamma _{A}(x)\leq \alpha
\} $
\end{center}

\QTP{Body Math}
is called an intuitionistic level cut of $A.$

\begin{theorem}
For an AG-groupoid $S,$ the following statements are true$.$
\end{theorem}

$(i)$ $A_{\alpha }$ is a right (left, two-sided) ideal of $S$ if $A$ is an
intuitionistic fuzzy right (left) ideal of $S$ but the converse is not true
in general.

$(ii)$ $A_{\alpha }$ is a bi-(generalized bi-) ideal of $S$ if $A$ is an
intuitionistic fuzzy bi-(generalized bi-) ideal of $S$ but the converse is
not true in general.

\begin{proof}
$(i)$: Let $S$ be an AG-groupoid and let $A$ be an intuitionistic fuzzy
right ideal of $S.$ If $x,y\in S\ $such that $x\in A_{\alpha },$ then $\mu
_{A}(x)\geq \alpha $ and $\gamma _{A}(x)\leq \alpha $ therefore $\mu
_{A}(xy)\geq \mu _{A}(x)\geq \alpha $ and $\gamma _{A}(xy)\leq \gamma
_{A}(x)\leq \alpha $. Thus $xy\in A_{\alpha },$ which shows that $A_{\alpha
} $ is a right ideal of $S.$ Let $y\in A_{\alpha },$ then $\mu _{A}(y)\geq
\alpha $ and $\gamma _{A}(y)\leq \alpha .$ If $A$ is an intuitionistic fuzzy
left ideal of $S,$ then $\mu _{A}(xy)\geq \mu _{A}(y)\geq \alpha $ and $%
\gamma _{A}(xy)\leq \gamma _{A}(y)\leq \alpha $ implies that $xy\in
A_{\alpha },$ which shows that $A_{\alpha }$ is a left ideal of $S.$

Conversely, let us define an $IFS$ $A=(\mu _{A},\gamma _{A})$ of an
AG-groupoid $S$ in Example \ref{ex} as follows: $\mu _{A}(1)=0.4,$ $\mu
_{A}(2)=0.8,$ $\mu _{A}(3)=\mu _{A}(4)=\mu _{A}(5)=0,$ $\gamma _{A}(1)=0.4,$ 
$\gamma _{A}(2)=0.3,$ $\gamma _{A}(3)=\gamma _{A}(4)=0.9$ and $\gamma
_{A}(5)=1.$ Let $\alpha =0.4,$ then it is easy to see that $A_{\alpha
}=\{a,b\}$ and one can easily verify from Example \ref{ex} that $\{a,b\}$ is
a right (left) ideal of $S$ but $\mu _{A}(21)\ngeq \mu _{A}(2)$ $(\gamma
_{A}(21)\nleq \gamma _{A}(2))$ and $\mu _{A}(12)\ngeq \mu _{A}(2)$ $(\gamma
_{A}(12)\nleq \gamma _{A}(2))$ implies that $A$ is not an intuitionistic
fuzzy right (left) ideal of $S.$

$(ii)$: Let $S$ be an AG-groupoid and let $A$ be an intuitionistic fuzzy
bi-(generalized bi-) ideal of $S.$ If $x,y$ and $z\in S$ such that $x$ and $%
z\in A_{\alpha },$ then $\mu _{A}(x)\geq \alpha ,$ $\gamma _{A}(x)\leq
\alpha $, $\mu _{A}(z)\geq \alpha $ and $\gamma _{A}(z)\leq \alpha $.
Therefore $\mu _{A}((xy)z)\geq \mu _{A}(x)\wedge \mu _{A}(z)\geq \alpha $
and $\gamma _{A}((xy)z)\leq \gamma _{A}(x)\vee \gamma _{A}(z)\leq \alpha $
implies that $(xy)z\in A_{\alpha }.$ Which shows that $A_{\alpha }$ is a
generalized bi-ideal of $S$. Now let $x,y\in A_{\alpha },$ then $\mu
_{A}(x)\geq \alpha ,$ $\gamma _{A}(x)\leq \alpha ,$ $\mu _{A}(y)\geq \alpha $
and $\gamma _{A}(y)\leq \alpha $. Therefore $\mu _{A}(xy)\geq \mu
_{A}(x)\wedge \mu _{A}(y)\geq \alpha $ and $\gamma _{A}(xy)\leq \gamma
_{A}(x)\vee \gamma _{A}(y)\leq \alpha $ implies that $xy\in A_{\alpha }.$
Thus $A_{\alpha }$ is a bi-ideal of $S$.

Conversely, let us define an $IFS$ $A=(\mu _{A},\gamma _{A})$ of an
AG-groupoid $S$ as in $(i).$ Then it is easy to observe that $A_{\alpha
}=\{a,b\}$ is a bi-(generalized bi-) ideal of $S$ but $\mu _{A}((ba)b)\ngeq
\mu _{A}(b)$ and $\gamma _{A}((ba)b)\nleq \gamma _{A}(b)$ implies that $A$
is not an intuitionistic fuzzy bi-(generalized bi-) ideal of $S$.
\end{proof}

Let $A=(\mu _{A},\gamma _{A})$ and $B=(\mu _{B},\gamma _{B})$ be any two $%
IFSs$ of an AG-groupoid $S$, then the product $A\circ B$ is defined by,

\begin{equation*}
\left( \mu _{A}\circ \mu _{B}\right) (a)=\left\{ 
\begin{array}{c}
\dbigvee\limits_{a=bc}\left\{ \mu _{A}(b)\wedge \mu _{B}(c)\right\} \text{,
if }a=bc\text{ for some }b,\text{ }c\in S. \\ 
0,\text{ otherwise.}%
\end{array}%
\right.
\end{equation*}

\begin{equation*}
\left( \gamma _{A}\circ \gamma _{B}\right) (a)=\left\{ 
\begin{array}{c}
\dbigwedge\limits_{a=bc}\left\{ \gamma _{A}(b)\vee \gamma _{B}(c)\right\} 
\text{, if }a=bc\text{ for some }b,\text{ }c\in S. \\ 
1,\text{ otherwise.}%
\end{array}%
\right.
\end{equation*}

$A\subseteq B$ means that

\begin{center}
$\mu _{A}(x)\leq \mu _{B}(x)$ and $\gamma _{A}(x)\geq \gamma _{B}(x)$ for
all $x$ in $S.$
\end{center}

\begin{lemma}
$($\cite{Mordeson},\cite{mad}$)$ \label{as}Let $S$ be an AG-groupoid$,$ then
the following holds.
\end{lemma}

$(i)$ An $IFS$ $A=(\mu _{A},\gamma _{A})$ is an intuitionistic fuzzy
AG-subgroupoid of $S$ if and only if $\mu _{A}\circ \mu _{A}\subseteq \mu
_{A}$ and $\gamma _{A}\circ \gamma _{A}\supseteq \gamma _{A}.$

$(ii)$ An $IFS$ $A=(\mu _{A},\gamma _{A})$ is intuitionistic fuzzy left
(right) ideal of $S$ if and only if $S\circ \mu _{A}\subseteq \mu _{A}$ and $%
\Theta \circ \gamma _{A}\supseteq \gamma _{A}$ $(\mu _{A}\circ S\subseteq
\mu _{A}$ and $\gamma _{A}\circ \Theta \supseteq \gamma _{A}).$

\begin{theorem}
Let $A=(\mu _{A},\gamma _{A})$ be an $IFS$ of an intra-regular AG-groupoid $%
S $ with left identity$,$ then the following conditions are equivalent.
\end{theorem}

$(i)$ $A=(\mu _{A},\gamma _{A})$ is an intuitionistic fuzzy bi-ideal of $S$.

$(ii)$ $(A\circ \delta )\circ A=A$ and $A\circ A=A,$ where $\delta
=(S_{\delta },\Theta _{\delta }).$

\begin{proof}
$(i)\Longrightarrow (ii):$ Let $A=(\mu _{A},\gamma _{A})$ be an
intuitionistic fuzzy bi-ideal of an intra-regular AG-groupoid $S$ with left
identity. Let $a\in A$, then there exists $x,$ $y\in S$ such that $%
a=(xa^{2})y.$ Now by using $(4),$ $(1),$ $(3)$ and $(2),$ we have%
\begin{eqnarray*}
a &=&(xa^{2})y=(x(aa))y=(a(xa))y=(y(xa))a=(y(x((x(aa))y)))a \\
&=&((ey)(x((a(xa))y)))a=((((a(xa))y)x)(ye))a=(((xy)(a(xa)))(ye))a \\
&=&((a((xy)(xa)))(ye))a=((a(x^{2}(ya)))(ye))a=(((ye)(x^{2}(ya)))a)a \\
&=&(((ye)(x^{2}(y((xa^{2})(ey)))))a)a=(((ye)(x^{2}(y((ye)(a^{2}x)))))a)a \\
&=&(((ye)(x^{2}(y(a^{2}((ye)x)))))a)a=(((ye)(x^{2}(a^{2}(y((ye)x)))))a)a \\
&=&(((ye)(a^{2}(x^{2}(y((ye)x)))))a)a=(((aa)((ye)(x^{2}(y((ye)x)))))a)a \\
&=&((((x^{2}(y((ye)x)))(ye))(aa))a)a=((a(((x^{2}(y((ye)x)))(ye))a))a)a=(pa)a
\end{eqnarray*}

where $p=a(((x^{2}(y((ye)x)))(ye))a).$ Therefore%
\begin{eqnarray*}
((\mu _{A}\circ S_{\delta })\circ \mu _{A})(a)
&=&\dbigvee\limits_{a=(pa)a}\left\{ (\mu _{A}\circ S_{\delta })(pa)\wedge
\mu _{A}(a)\right\} \\
&\geq &\dbigvee\limits_{pa=pa}\left\{ \mu _{A}(p)\circ S_{\delta
}(a)\right\} \wedge \mu _{A}(a) \\
&\geq &\left\{ \mu _{A}(a(((x^{2}(y((ye)x)))(ye))a))\wedge S_{\delta
}(a)\right\} \wedge \mu _{A}(a) \\
&\geq &\mu _{A}(a)\wedge 1\wedge \mu _{A}(a)=\mu _{A}(a)
\end{eqnarray*}

and%
\begin{eqnarray*}
((\gamma _{A}\circ \Theta _{\delta })\circ \gamma _{A})(a)
&=&\dbigwedge\limits_{a=(pa)a}\left\{ (\gamma _{A}\circ \Theta _{\delta
})(pa)\vee \gamma _{A}(a)\right\} \\
&\leq &\dbigwedge\limits_{pa=pa}\left\{ \gamma _{A}(p)\circ \Theta _{\delta
}(a)\right\} \vee \gamma _{A}(a) \\
&\leq &\left\{ \gamma _{A}(a(((x^{2}(y((ye)x)))(ye))a))\vee \Theta _{\delta
}(a)\right\} \vee \gamma _{A}(a) \\
&\leq &\gamma _{A}(a)\vee 1\vee \gamma _{A}(a)=\gamma _{A}(a).
\end{eqnarray*}

This shows that $(\mu _{A}\circ S_{\delta })\circ \mu _{A}\supseteq \mu _{A}$
and $(\gamma _{A}\circ \Theta _{\delta })\circ \gamma _{A}\subseteq \gamma
_{A},$ which implies that $(A\circ \delta )\circ A\supseteq A.$ Now by using 
$(4),$ $(1)$ and $(3),$ we have%
\begin{eqnarray*}
a &=&(xa^{2})y=(x(aa))y=(a(xa))y=(y(xa))a=(y(x((xa^{2})(ey))))a \\
&=&(y(x((ye)(a^{2}x))))a=(y(x(a^{2}((ye)x))))a=(y(a^{2}(x((ye)x))))a \\
&=&((aa)(y(x((ye)x))))a=(((x((ye)x))y)(aa))a=(a((x((ye)x))a))a=(ap)a
\end{eqnarray*}

where $p=((x((ye)x))y)a.$ Therefore%
\begin{eqnarray*}
((\mu _{A}\circ S_{\delta })\circ \mu _{A})(a)
&=&\dbigvee\limits_{a=(ap)a}\left\{ (\mu _{A}\circ S_{\delta })(ap)\wedge
\mu _{A}(a)\right\} \\
&=&\dbigvee\limits_{a=(ap)a}\left( \dbigvee\limits_{ap=ap}\mu _{A}(a)\wedge
S_{\delta }(p)\right) \wedge \mu _{A}(a) \\
&=&\dbigvee\limits_{a=(ap)a}\left\{ \mu _{A}(a)\wedge 1\wedge \mu
_{A}(a)\right\} =\dbigvee\limits_{a=(ap)a}\mu _{A}(a)\wedge \mu _{A}(a) \\
&\leq &\dbigvee\limits_{a=(ap)a}\mu _{A}((a((x((ye)x))a))a)=\mu _{A}(a).
\end{eqnarray*}

This shows that $(\mu _{A}\circ S_{\delta })\circ \mu _{A}\subseteq \mu _{A}$
and similarly we can show that $(\gamma _{A}\circ \Theta _{\delta })\circ
\gamma _{A}\supseteq \gamma _{A},$ which implies that $(A\circ \delta )\circ
A\subseteq A.$ Thus $(A\circ \delta )\circ A=A.$ We have shown that $%
a=((a(((x^{2}(y((ye)x)))(ye))a))a)a.$ Let $a=pa$ where $%
p=(a(((x^{2}(y((ye)x)))(ye))a))a.$ Therefore%
\begin{eqnarray*}
(\mu _{A}\circ \mu _{A})(a) &=&\dbigvee\limits_{a=pa}\left\{ \mu
_{A}((a(((x^{2}(y((ye)x)))(ye))a))a)\wedge \mu _{A}(a)\right\} \\
&\geq &\mu _{A}(a)\wedge \mu _{A}(a)\wedge \mu _{A}(a)=\mu _{A}(a).
\end{eqnarray*}

This shows that $\mu _{A}\circ \mu _{A}\supseteq \mu _{A}$ and similarly we
can show that $\gamma _{A}\circ \gamma _{A}\subseteq \gamma _{A}.$ Now by
using Lemma \ref{as}, we get $A\circ A=A.$

$(ii)\Longrightarrow (i):$ Let $A=(\mu _{A},\gamma _{A})$ be an $IFS$ of an
intra-regular AG-groupoid $S$, then%
\begin{eqnarray*}
\mu _{A}((xa)y) &=&((\mu _{A}\circ S_{\delta })\circ \mu
_{A})((xa)y)=\dbigvee\limits_{(xa)y=(xa)y}\{(\mu _{A}\circ S_{\delta
})(xa)\wedge \mu _{A}(y)\} \\
&\geq &\dbigvee\limits_{xa=xa}\left\{ \mu _{A}(x)\wedge S_{\delta
}(a)\right\} \wedge \mu _{A}(y)\geq \mu _{A}(x)\wedge 1\wedge \mu
_{A}(y)=\mu _{A}(x)\wedge \mu _{A}(z).
\end{eqnarray*}

This shows that $\mu _{A}((xa)y)\geq \mu _{A}(x)\wedge \mu _{A}(z)$ and
similarly we can show that $\gamma _{A}((xa)y)\leq \gamma _{A}(x)\vee \gamma
_{A}(z)$. Also by Lemma \ref{as}, $A$ is an intuitionistic fuzzy
AG-subgroupoid of $S$ and therefore $A$ is an intuitionistic fuzzy bi-ideal
of $S.$
\end{proof}

\begin{theorem}
Let $A=(\mu _{A},\gamma _{A})$ be an $IFS$ of an intra-regular AG-groupoid $%
S $ with left identity$,$ then the following conditions are equivalent.
\end{theorem}

$(i)$ $A=(\mu _{A},\gamma _{A})$ is an intuitionistic fuzzy interior ideal
of $S$.

$(ii)$ $(\delta \circ A)\circ \delta =A,$ where $\delta =(S_{\delta },\Theta
_{\delta }).$

\begin{proof}
$(i)\Longrightarrow (ii):$ Let $A=(\mu _{A},\gamma _{A})$ be an
intuitionistic fuzzy interior ideal of an intra-regular AG-groupoid $S$ with
left identity. Let $a\in A$, then there exists $x,$ $y\in S$ such that $%
a=(xa^{2})y.$ Now by using $(4),$ $(3)$and $(1),$ we have%
\begin{equation*}
a=(x(aa))y=(a(xa))y=((ea)(xa))y=((ax)(ae))y=(((ae)x)a)y.
\end{equation*}

Therefore%
\begin{eqnarray*}
((S_{\delta }\circ \mu _{A})\circ S_{\delta })(a)
&=&\dbigvee\limits_{a=(((ae)x)a)y}\left\{ (S_{\delta }\circ \mu
_{A})(((ae)x)a)\wedge S_{\delta }(y)\right\} \\
&\geq &\dbigvee\limits_{((ae)x)a=((ae)x)a}\left\{ (S_{\delta }((ae)x)\wedge
\mu _{A}(a)\right\} \wedge 1 \\
&\geq &1\wedge \mu _{A}(a)\wedge 1=\mu _{A}(a).
\end{eqnarray*}

This proves that $(S_{\delta }\circ \mu _{A})\circ S_{\delta }\supseteq \mu
_{A}$ and similarly we can show that $(\Theta _{\delta }\circ \gamma
_{A})\circ \Theta _{\delta }\subseteq \gamma _{A}$, therefore $(\delta \circ
A)\circ \delta \supseteq A.$ Now again%
\begin{eqnarray*}
((S_{\delta }\circ \mu _{A})\circ S_{\delta })(a)
&=&\dbigvee\limits_{a=(xa^{2})y}\left\{ (S_{\delta }\circ \mu
_{A})(xa^{2})\wedge S_{\delta }(y)\right\} \\
&=&\dbigvee\limits_{a=(xa^{2})y}\left(
\dbigvee\limits_{xa^{2}=xa^{2}}S_{\delta }(x)\wedge \mu _{A}(a^{2})\right)
\wedge S_{\delta }(y) \\
&=&\dbigvee\limits_{a=(xa^{2})y}\left\{ 1\wedge \mu _{A}(a^{2})\wedge
1\right\} =\dbigvee\limits_{a=(xa^{2})y}\mu _{A}(a^{2}) \\
&\leq &\dbigvee\limits_{a=(xa^{2})y}\mu _{A}((xa^{2})y)=\mu _{A}(a).
\end{eqnarray*}

Thus $(S_{\delta }\circ \mu _{A})\circ S_{\delta }\subseteq \mu _{A}$ and
similarly we can show that $(\Theta _{\delta }\circ \gamma _{A})\circ \Theta
_{\delta }\supseteq \gamma _{A},$ therefore $(\delta \circ A)\circ \delta
\subseteq A.$ Hence it follows that $(\delta \circ A)\circ \delta =A.$

$(ii)\Longrightarrow (i):$ Let $A=(\mu _{A},\gamma _{A})$ be an $IFS$ of an
intra-regular AG-groupoid $S$, then%
\begin{eqnarray*}
\mu _{A}((xa)y) &=&((S_{\delta }\circ \mu _{A})\circ S_{\delta
})((xa)y)=\dbigvee\limits_{(xa)y=(xa)y}\left\{ (S_{\delta }\circ \mu
_{A})(xa)\wedge S_{\delta }(y)\right\} \\
&\geq &\dbigvee\limits_{xa=xa}\left\{ (S_{\delta }(x)\circ \mu
_{A}(a)\right\} \wedge S_{\delta }(y)\geq \mu _{A}(a).
\end{eqnarray*}

Similarly we can show that $\gamma _{A}((xa)y)\leq \gamma _{A}(a)$ and
therefore $A=(\mu _{A},\gamma _{A})$ is an intuitionistic fuzzy interior
ideal of $S.$
\end{proof}

\begin{lemma}
\label{iff}Let $A=(\mu _{A},\gamma _{A})$ be an $IFS$ of an intra-regular
AG-groupoid $S$ with left identity$,$ then $A=(\mu _{A},\gamma _{A})$ is an
intuitionistic fuzzy left ideal of $S$ if and only if $A=(\mu _{A},\gamma
_{A})$ is an intuitionistic fuzzy right ideal of $S.$
\end{lemma}

\begin{proof}
Let $S$ be an intra-regular AG-groupoid and let $A=(\mu _{A},\gamma _{A})$
be an intuitionistic fuzzy left ideal of $S.$ Now for $a,b\in S$ there
exists $x,y,x^{^{\prime }},y^{^{\prime }}\in S$ such that $a=(xa^{2})y$ and $%
b=(x^{^{\prime }}b^{2})y^{^{\prime }},$ then by using $(1),$ $(3)$ and $(4),$
we have 
\begin{eqnarray*}
\mu _{A}(ab) &=&\mu _{A}(((xa^{2})y)b)=\mu _{A}((by)(x(aa)))=\mu
_{A}(((aa)x)(yb)) \\
&=&\mu _{A}(((xa)a)(yb))=\mu _{A}(((xa)(ea))(yb))=\mu _{A}(((ae)(ax))(yb)) \\
&=&\mu _{A}((a((ae)x))(yb))=\mu _{A}(((yb)((ae)x))a)\geq \mu _{A}(a).
\end{eqnarray*}

Similarly we can get $\gamma _{A}(ab)\leq \gamma _{A}(a),$ which implies
that $A=(\mu _{A},\gamma _{A})$ is an intuitionistic fuzzy right ideal of $%
S. $

Conversely let $A=(\mu _{A},\gamma _{A})$ be an intuitionistic fuzzy right
ideal of $S.$ Now by using $(4)$ and $(3),$ we have 
\begin{eqnarray*}
\mu _{A}(ab) &=&\mu _{A}(a((x^{^{\prime }}b^{2})y^{^{\prime }})=\mu
_{A}((x^{^{\prime }}b^{2})(ay^{^{\prime }}))=\mu _{A}((y^{^{\prime
}}a)(b^{2}x^{^{\prime }})) \\
&=&\mu _{A}(b^{2}((y^{^{\prime }}a)x))\geq \mu _{A}(b).
\end{eqnarray*}

Also we can get $\gamma _{A}(ab)\leq \gamma _{A}(b),$ which implies that $%
A=(\mu _{A},\gamma _{A})$ is an intuitionistic fuzzy left ideal of $S.$
\end{proof}

An AG-groupoid $S$ is called a left (right) duo if every left (right) ideal
of $S$ is a two-sided ideal of $S$ and is called a duo if it is both a left
and a right duo.

An AG-groupoid $S$ is called an intuitionistic fuzzy left (right) duo if
every intuitionistic fuzzy left (right) ideal of $S$ is an intuitionistic
fuzzy two-sided ideal of $S$ and is called an intuitionistic fuzzy duo if it
is both an intuitionistic fuzzy left and an intuitionistic fuzzy right duo.

\begin{corollary}
Every intra-regular AG-groupoid with left identity is an intuitionistic
fuzzy duo.
\end{corollary}

Let $S$ be an AG-groupoid and let $\emptyset \neq A\subseteq S$ be an $IFS$
of $S,$ then the intuitionistic characteristic function $\chi _{A}=(\mu
_{\chi _{A}},\gamma _{\chi _{A}})$ of $A$ is defined as

\begin{center}
$\mu _{\chi _{A}}(x)=\left\{ 
\begin{array}{c}
1\text{, if }x\in A \\ 
0\text{, if }x\notin A%
\end{array}%
\right. $ and $\gamma _{\chi _{A}}(x)=\left\{ 
\begin{array}{c}
0\text{, if }x\in A \\ 
1\text{, if }x\notin A%
\end{array}%
\right. $
\end{center}

\QTP{Body Math}
It is clear that $\gamma _{\chi _{A}}$ acts as a complement of $\mu _{\chi
_{A}},$ that is, $\gamma _{\chi _{A}}=\mu _{\chi _{A^{C}}}.$

\begin{lemma}
$($\cite{Mordeson},\cite{mad}$)$ \label{00} For any subset $A$ of an
AG-groupoid $S,$ the following properties holds.
\end{lemma}

$(i)$ $A$ is an AG-subgroupoid of $S$ if and only if $\chi _{A}$ is an
intuitionistic fuzzy AG-subgroupoid of $S$.

$(ii)$ $A$ is a fuzzy left (right, two-sided) ideal of $S$ if and only if $%
\chi _{A}$ is an intuitionistic fuzzy left (right, two-sided) ideal of $S$.

\begin{theorem}
An intra-regular AG-groupoid $S$ with left identity is a left (right) duo if
and only if it is an intuitionistic fuzzy left (right) duo.
\end{theorem}

\begin{proof}
Let an intra-regular AG-groupoid $S$ be a left duo and let $A=(\mu
_{A},\gamma _{A})$ be any intuitionistic fuzzy left ideal of $S$. Let $%
a,b\in S,$ then $a\in (Sa^{2})S.$ Now as $Sa$ is a left ideal of $S,$
therefore by hypothesis, $Sa$ is a two-sided ideal of $S.$ Now by u$\sin $g $%
(4)$ and $(1),$ we have 
\begin{equation*}
ab\in ((Sa^{2})S)b=((S(aa))S)b=((a(Sa))S)b=((S(Sa))a)b\subseteq
((S(Sa))S)S\subseteq Sa.
\end{equation*}

Thus $ab=ca$ for some $c\in S.$ Now $\mu _{A}(ab)=\mu _{A}(ca)\geq \mu
_{A}(a)$ and similarly $\gamma _{A}(ab)=\gamma _{A}(ca)\leq \gamma _{A}(a)$
implies that $A$ is an intuitionistic fuzzy right ideal of $S$ and therefore 
$S$ is an intuitionistic fuzzy left duo.

Conversely, assume that $S$ is a fuzzy left duo and $L$ is any left ideal of 
$S.$ Now by Lemma \ref{00}, the intuitionistic characteristic function$\
\chi _{L}=(\mu _{\chi _{L}},\gamma _{\chi _{L}})$ of$\,$\ $L$ is an
intuitionistic fuzzy left ideal of $S$. Thus by hypothesis $\chi _{L}$ is an
intuitionistic fuzzy two-sided ideal of $S$ and by using Lemma \ref{00}, $L$
is a two-sided ideal of $S$. Thus $S$ is a left duo.

Now again let $S$ be an intra-regular AG-groupoid such that $S$ is a right
duo and assume that $A=(\mu _{A},\gamma _{A})$ is any fuzzy right ideal of $%
S $. Clearly $b^{2}S$ is a right ideal and so is a two-sided ideal of $S.$
Let $a,b\in S,$ then there exist $x,y\in S$ such that $b=(xb^{2})y$. Now by
using $(3)$, we have 
\begin{equation*}
ab=a((xb^{2})y)=a(((ex)(eb^{2}))y)=a(((b^{2}e)(xe))y)\subseteq
S(((b^{2}S)S)S)\subseteq b^{2}S.
\end{equation*}

Thus $ab=(bb)c$ for some $c\in S.$ Now $\mu _{A}(ab)=\mu _{A}((bb)c)\geq \mu
_{A}(b)$ and $\gamma _{A}(ab)=\gamma _{A}((bb)c)\leq \gamma _{A}(b)$ implies
that $A$ is an intuitionistic fuzzy left ideal of $S$ and therefore $S$ is
an intuitionistic fuzzy right duo. The Converse is simple.
\end{proof}

\begin{lemma}
\label{aqw}In an intra-regular AG-groupoid $S,$ $\delta \circ A=A$ and $%
A\circ \delta =A$ holds for an $IFS$ $A=(\mu _{A},\gamma _{A})$ of $S$ where 
$\delta =(S_{\delta },\Theta _{\delta }).$
\end{lemma}

\begin{proof}
Let $A=(\mu _{A},\gamma _{A})$ be an $IFS$ of an intra-regular AG-groupoid $%
S $ and let $a\in S,$ then there exist $x\in S$ such that $a=(xa^{2})y.$ Now
by using $(4)$ and $(1),$ we have%
\begin{equation*}
a=(x(aa))y=(a(xa))y=(y(xa))a.
\end{equation*}

Therefore%
\begin{eqnarray*}
(S_{\delta }\circ \mu _{A})(a) &=&\dbigvee\limits_{a=(y(xa))a}\{S_{\delta
}(y(xa))\wedge \mu _{A}(a)\}=\dbigvee\limits_{a=(y(xa))a}\{1\wedge \mu
_{A}(a)\} \\
&=&\dbigvee\limits_{a=(y(xa))a}\mu _{A}(a)=\mu _{A}(a).
\end{eqnarray*}

Similarly we can show that $\Theta _{\delta }\circ \gamma _{A}=\gamma _{A},$
which shows that $\delta \circ A=A.$

Now by using $(3)$ and $(4),$ we have%
\begin{equation*}
a=(xa^{2})(ey)=(ye)(a^{2}x)=(aa)((ye)x)=(x(ye))(aa)=a((x(ye))a).
\end{equation*}

Therefore%
\begin{eqnarray*}
(\mu _{A}\circ S_{\delta })(a) &=&\dbigvee\limits_{a=a((x(ye))a)}\{\mu
_{A}(a)\wedge S_{\delta }((x(ye))a)\}=\dbigvee\limits_{a=a((x(ye))a)}\{\mu
_{A}(a)\wedge 1\} \\
&=&\dbigvee\limits_{a=a((x(ye))a)}\mu _{A}(a)=\mu _{A}(a).
\end{eqnarray*}

Similarly we can show that $\gamma _{A}\circ \Theta _{\delta }=\gamma _{A}$
which shows that $A\circ \delta =A.$
\end{proof}

\begin{corollary}
\label{cor}In an intra-regular AG-groupoid $S,$ $\delta \circ A=A$ and $%
A\circ \delta =A$ holds for every intuitionistic fuzzy left $($right$,$
two-sided$)$ $A=(\mu _{A},\gamma _{A})$ of $S,$ where $\delta =(S_{\delta
},\Theta _{\delta }).$
\end{corollary}

\begin{lemma}
\label{ss}In an intra-regular AG-groupoid $S,$ $\delta \circ \delta =\delta
, $ where $\delta =(S_{\delta },\Theta _{\delta }).$
\end{lemma}

\begin{proof}
Let $S$ be an intra-regular AG-groupoid, then%
\begin{equation*}
(S_{\delta }\circ S_{\delta })(a)=\dbigvee\limits_{a=(xa^{2})y}\{S_{\delta
}(xa^{2})\wedge S_{\delta }(y)\}=1=S_{\delta }(a)
\end{equation*}

and%
\begin{equation*}
(\Theta _{\delta }\circ \Theta _{\delta
})(a)=\dbigwedge\limits_{a=(xa^{2})y}\{\Theta _{\delta }(xa^{2})\vee \Theta
_{\delta }(y)\}=0=\Theta _{\delta }(a).
\end{equation*}
\end{proof}

\begin{theorem}
\label{ae}Let $A=(\mu _{A},\gamma _{A})$ be an $IFS$ of an intra-regular
AG-groupoid $S$ with left identity$,$ then the following conditions are
equivalent.
\end{theorem}

$(i)$ $A=(\mu _{A},\gamma _{A})$ is an intuitionistic fuzzy quasi ideal of $%
S $.

$(ii)$ $(A\circ \delta )\cap (\delta \circ A)=A,$ where $\delta =(S_{\delta
},\Theta _{\delta }).$

\begin{proof}
$(i)\Longrightarrow (ii)$ can be followed from Lemma \ref{aqw} and $%
(ii)\Longrightarrow (i)$ is obvious.
\end{proof}

\begin{theorem}
\label{qu}Let $A=(\mu _{A},\gamma _{A})$ an $IFS$ of an intra-regular
AG-groupoid $S$ with left identity, then the following statements are
equivalent.
\end{theorem}

$(i)$ $A$ is an intuitionistic fuzzy two-sided ideal of $S.$

$(ii)$ $A$ is an intuitionistic fuzzy quasi ideal of $S.$

\begin{proof}
$(i)\Longrightarrow (ii)$ is an easy consequence of Corollary \ref{cor} and
Theorem \ref{ae}.

$(ii)\Longrightarrow (i):$ Let $A=(\mu _{A},\gamma _{A})$ be an
intuitionistic fuzzy quasi ideal of an intra-regular AG-groupoid $S$ with
left identity and let $a\in S,$ then there exist $x\in S$ such that $%
a=(xa^{2})y.$ Now by using $(4)$ and $(3),$ we have%
\begin{eqnarray*}
a &=&(x(aa))y=(a(xa))(ey)=(ye)((xa)(ea))=(ye)((ae)(ax)) \\
&=&(ye)(a((ae)x))=a((ye)((ae)x)).
\end{eqnarray*}

Therefore 
\begin{eqnarray*}
(\mu _{A}\circ S_{\delta })(a) &=&\dbigvee\limits_{a=a((ye)((ae)x))}\{\mu
_{A}(a)\wedge S_{\delta }((ye)((ae)x))\} \\
&\geq &\mu _{A}(a)\wedge 1=\mu _{A}(a)\text{.}
\end{eqnarray*}%
Similarly we can show that $\gamma _{A}\circ \Theta _{\delta }\subseteq
\gamma _{A}$ which implies that $A\circ \delta \supseteq A.$ Now by using
Lemmas \ref{aqw}, \ref{ss} and $(2),$ we have%
\begin{equation*}
A\circ \delta =(\delta \circ A)\circ (\delta \circ \delta )=(\delta \circ
\delta )\circ (A\circ \delta )=\delta \circ (A\circ \delta )\supseteq \delta
\circ A\text{.}
\end{equation*}

This shows that $\delta \circ A\subseteq (A\circ \delta )\cap (\delta \circ
A).$ As $A$ is an intuitionistic fuzzy quasi ideal of $S,$ thus we get $%
\delta \circ A\subseteq A$. Now by using Lemma \ref{as}, $A$ is an
intuitionistic fuzzy left ideal of $S$ and by Lemma \ref{iff}, $A$ is an
intuitionistic fuzzy right ideal of $S,$ that is, $A\ $is an intuitionistic
fuzzy two-sided ideal of $S.$
\end{proof}

\begin{theorem}
\label{intr}Let $A=(\mu _{A},\gamma _{A})$ be an $IFS$ of an intra-regular
AG-groupoid $S$ with left identity, then the following statements are
equivalent.
\end{theorem}

$(i)$ $A$ is an intuitionistic fuzzy two-sided ideal of $S.$

$(ii)$ $A$ is an intuitionistic fuzzy interior ideal of $S.$

\begin{proof}
$(i)\Longrightarrow (ii)$ is obvious.

$(ii)\Longrightarrow (i):$ Let $A=(\mu _{A},\gamma _{A})$ be an
intuitionistic fuzzy interior ideal of an intra-regular AG-groupoid $S$ with
left identity and let $a,b\in S,$ then there exist $x\in S$ such that $%
a=(xa^{2})y.$ Now by using $(4)$, $(1)$ and $(3),$ we have%
\begin{eqnarray*}
\mu _{A}(ab) &=&\mu _{A}(((x(aa))y)b)=\mu _{A}(((a(xa))y)b)=\mu
_{A}((by)(a(xa))) \\
&=&\mu _{A}(((xa)a)(yb))\geq \mu _{A}(a).
\end{eqnarray*}

Similarly we can prove that $\gamma _{A}(ab)\leq \gamma _{A}(a).$ Thus $A$
is an intuitionistic fuzzy right ideal of $S$ and by using Lemma \ref{iff}, $%
A$ is an intuitionistic fuzzy two-sided ideal of $S.$
\end{proof}

\begin{theorem}
Let $A=(\mu _{A},\gamma _{A})$ be an $IFS$ of an intra-regular AG-groupoid $%
S $ with left identity, then the following statements are equivalent.
\end{theorem}

$(i)$ $A$ is an intuitionistic fuzzy left ideal of $S$.

$(ii)$ $A$ is an intuitionistic fuzzy right ideal of $S$.

$(iii)$ $A$ is an intuitionistic fuzzy two-sided ideal of $S$.

$(iv)$ $A$ is an intuitionistic fuzzy bi-ideal of $S$.

$(v)$ $A$ is an intuitionistic fuzzy generalized bi-ideal of $S$.

$(vi)$ $A$ is an intuitionistic fuzzy interior ideal of $S$.

$(vii)$ $A$ is an intuitionistic fuzzy quasi ideal of $S.$

$(viii)$ $A\circ \delta =A$ and $\delta \circ A=A.$

\begin{proof}
$(i)\Longrightarrow (viii)$ can be followed from Corollary \ref{cor} and $%
(ix)\Longrightarrow (viii)$ is obvious.

$(vii)\Longrightarrow (vi):$ Let $A=(\mu _{A},\gamma _{A})$ be an
intuitionistic fuzzy quasi ideal of an intra-regular AG-groupoid $S$ with
left identity. Now for $a\in S$ there exist $x,y\in S$ such that $%
a=(ba^{2})c.$ Now by using $(4),$ $(3)$ and $(1)$, we have%
\begin{eqnarray*}
(xa)y &=&(x((ba^{2})c))y=((ba^{2})(xc))y=((cx)(a^{2}b))y=(a^{2}((cx)b))y \\
&=&(y((cx)b))(aa)=a((y((cx)b))a)
\end{eqnarray*}

and%
\begin{eqnarray*}
(xa)y &=&(x((ba^{2})c))y=((ba^{2})(xc))y=((cx)(a^{2}b))y=(a^{2}((cx)b))y \\
&=&(y((cx)b))(aa)=(aa)(((cx)b)y)=((((cx)b)y)a)a.
\end{eqnarray*}
Now by using Theorem \ref{ae}, we have%
\begin{equation*}
\mu _{A}((xa)y)=((\mu _{A}\circ S_{\delta })\cap (S_{\delta }\circ \mu
_{A}))((xa)y)=(\mu _{A}\circ S_{\delta })((xa)y)\wedge (S_{\delta }\circ \mu
_{A})((xa)y).
\end{equation*}

Now%
\begin{equation*}
(\mu _{A}\circ S_{\delta
})((xa)y)=\dbigvee\limits_{(xa)y=a((y((cx)b))a)}\left\{ \mu _{A}(a)\wedge
S_{\delta }((y((cx)b))a)\right\} \geq \mu _{A}(a)
\end{equation*}

and 
\begin{equation*}
\left( S_{\delta }\circ \mu _{A}\right)
((xa)y)=\dbigvee\limits_{(xa)y=((((cx)b)y)a)a}\left\{ S_{\delta
}((((cx)b)y)a)\wedge \mu _{A}(a)\right\} \geq \mu _{A}(a).
\end{equation*}

This implies that $\mu _{A}((xa)y)\geq \mu _{A}(a)$ and similarly we can
show that $\gamma _{A}((xa)y)\leq \gamma _{A}(a)$. Thus $A$ is an
intuitionistic fuzzy interior ideal of $S.$

$(vi)\Longrightarrow (v):$ Let $A$ be an intuitionistic fuzzy interior ideal
of $S,$ then by Theorem \ref{intr}, $A$ is an intuitionistic fuzzy two-sided
ideal of $S$ and it is easy to observe that $A$ is an intuitionistic fuzzy
generalized bi-ideal of $S$.

$(v)\Longrightarrow (iv)$: Let $A=(\mu _{A},\gamma _{A})$ be an
intuitionistic fuzzy generalized bi-ideal of an intra-regular AG-groupoid $S$
with left identity . Let $a\in S$, then there exists $x,y\in S$ such that $%
a=(xa^{2})y.$ Now by using $(4),$ $(3)$ and $(1),$ we have 
\begin{eqnarray*}
\mu _{A}(ab) &=&\mu _{A}(((x(aa))y)b)=\mu _{A}((((ea)(xa))y)b)=\mu
_{A}((((ax)(ae))y)b) \\
&=&\mu _{A}(((a((ax)e))(ey))b)=\mu _{A}(((ye)(((ax)e)a))b) \\
&=&\mu _{A}(((ye)((ae)(ax)))b)=\mu _{A}(((ye)(a((ae)x)))b) \\
&=&\mu _{A}((a((ye)((ae)x)))b)\geq \mu _{A}(a)\wedge \mu _{A}(b).
\end{eqnarray*}

Similarly we can show that $\gamma _{A}(ab)\leq \gamma _{A}(a)\vee \gamma
_{A}(b)$ and therefore $A=(\mu _{A},\gamma _{A})$ is an intuitionistic fuzzy
bi-ideal of $S$.

$(iv)\Longrightarrow (iii):$ Let $A=(\mu _{A},\gamma _{A})$ be an
intuitionistic fuzzy bi-ideal of an intra-regular AG-groupoid $S$ with left
identity. Let $a\in S$, then there exists $x,y\in S$ such that $a=(xa^{2})y.$
Now by using $(4),$ $(1)$ and $(3),$ we have%
\begin{eqnarray*}
\mu _{A}(ab) &=&\mu _{A}(((x(aa))y)b)=\mu _{A}(((a(xa))y)b)=\mu
_{A}((by)((ea)(xa))) \\
&=&\mu _{A}((by)((ax)(ae)))=\mu _{A}(((ae)(ax))(yb))=\mu _{A}((a((ae)x))(yb))
\\
&=&\mu _{A}(((yb)((ae)x))a)=\mu _{A}(((yb)((((xa^{2})y)e)x))a) \\
&=&\mu _{A}(((yb)((y(xa^{2}))(ex)))a)=\mu _{A}(((yb)((xe)((xa^{2})(ey))))a)
\\
&=&\mu _{A}(((yb)((xe)((ye)(a^{2}x))))a)=\mu
_{A}(((yb)((xe)(a^{2}((ye)x))))a) \\
&=&\mu _{A}(((yb)(a^{2}((xe)((ye)x))))a)=\mu
_{A}((a^{2}((yb)((xe)((ye)x))))a) \\
&\geq &\mu _{A}(a^{2})\wedge \mu _{A}(a)\geq \mu _{A}(a)\wedge \mu
_{A}(a)\wedge \mu _{A}(a)=\mu _{A}(a).
\end{eqnarray*}

Similarly we can prove that $\gamma _{A}(ab)\leq \gamma _{A}(a)$ and
therefore $A=(\mu _{A},\gamma _{A})$ is an intuitionistic fuzzy right ideal
of $S$. Now by using Lemma \ref{iff}, $A=(\mu _{A},\gamma _{A})$ is an
intuitionistic fuzzy two-sided ideal of $S$.

$(iii)\Longrightarrow (ii)$ and $(ii)\Longrightarrow (i)$ are an easy
consequences of Lemma \ref{iff}.
\end{proof}

Let $A=(\mu _{A},\gamma _{A})$ and $B=(\mu _{B},\gamma _{B})$ are $IFSs$ of
an AG-groupoid $S.$ The symbols $A\cap B$ will means the following $IFS$ of $%
S$

\begin{center}
$(\mu _{A}\cap \mu _{B})(x)=\min \{\mu _{A}(x),\mu _{B}(x)\}=\mu
_{A}(x)\wedge \mu _{B}(x),$ for all $x$ in $S.$

$(\gamma _{A}\cup \gamma _{B})(x)=\max \{\gamma _{A}(x),\gamma
_{B}(x)\}=\gamma _{A}(x)\vee \gamma _{B}(x),$ for all $x$ in $S.$
\end{center}

\QTP{Body Math}
The symbols $A\cup B$ will means the following $IFS$ of $S$

\begin{center}
$(\mu _{A}\cup \mu _{B})(x)=\max \{\mu _{A}(x),\mu _{B}(x)\}=\mu _{A}(x)\vee
\mu _{B}(x),$ for all $x$ in $S.$

$(\gamma _{A}\cap \gamma _{B})(x)=\min \{\gamma _{A}(x),\gamma
_{B}(x)\}=\gamma _{A}(x)\wedge \gamma _{B}(x),$ for all $x$ in $S.$
\end{center}

\begin{lemma}
\label{fgh}Let $S$ be an intra-regular AG-groupoid with left identity and
let $A=(\mu _{A},\gamma _{A})$ and $B=(\mu _{B},\gamma _{B})$ are any
intuitionistic fuzzy two-sided ideals of $S,$ then $A\circ B=A\cap B$.
\end{lemma}

\begin{proof}
Assume that $A=(\mu _{A},\gamma _{A})$ and $B=(\mu _{B},\gamma _{B})$ are
any intuitionistic fuzzy two-sided ideals of an intra-regular AG-groupoid $S$
with left identity, then by using Lemma \ref{as}, we have $\mu _{A}\circ \mu
_{B}\subseteq \mu _{A}\cap \mu _{B}$ and $\gamma _{A}\circ \gamma
_{B}\supseteq \gamma _{A}\cup \gamma _{B},$ which shows that $A\circ
B\subseteq A\cap B$. Let $a\in S,$ then there exists $x,y\in S$ such that $%
a=(xa^{2})y.$ Now by using $(4)$ and $(2),$ we have%
\begin{equation*}
a=(x(aa))y=(a(xa))(ey)=(ae)((xa)y).
\end{equation*}%
Therefore, we have%
\begin{eqnarray*}
(\mu _{A}\circ \mu _{B})(a) &=&\dbigvee\limits_{a=(ae)((xa)y)}\{\mu
_{A}(ae)\wedge \mu _{B}((xa)y)\}\geq \mu _{A}(ae)\wedge \mu _{B}((xa)y) \\
&\geq &\mu _{A}(a)\wedge \mu _{B}(a)=(\mu _{A}\cap \mu _{B})(a)
\end{eqnarray*}

and%
\begin{eqnarray*}
(\gamma _{A}\circ \gamma _{A})(a)
&=&\dbigwedge\limits_{a=(ae)((xa)y)}\left\{ \gamma _{A}(ae)\vee \gamma
_{A}((xa)y)\right\} \leq \gamma _{A}(ae)\vee \gamma _{A}((xa)y) \\
&\leq &\gamma _{A}(a)\vee \gamma _{A}(a)=(\gamma _{A}\cup \gamma _{A})(a).
\end{eqnarray*}

Thus we get that $\mu _{A}\circ \mu _{B}\supseteq \mu _{A}\cap \mu _{B}$ and 
$\gamma _{A}\circ \gamma _{B}\subseteq \gamma _{A}\cup \gamma _{B},$ which
give us $A\circ B\supseteq A\cap B$ and therefore $A\circ B=A\cap B.$
\end{proof}

The converse of Lemma \ref{fgh} is not true in general which is discussed in
the following.

Let us consider an AG-groupoid $S=\left\{ 1,2,3,4,5\right\} $ with left
identity $4$ in the following Cayley's table.

\begin{center}
\begin{tabular}{l|lllll}
. & $1$ & $2$ & $3$ & $4$ & $5$ \\ \hline
$1$ & $1$ & $1$ & $1$ & $1$ & $1$ \\ 
$2$ & $1$ & $5$ & $5$ & $3$ & $5$ \\ 
$3$ & $1$ & $5$ & $5$ & $2$ & $5$ \\ 
$4$ & $1$ & $2$ & $3$ & $4$ & $5$ \\ 
$5$ & $1$ & $5$ & $5$ & $5$ & $5$%
\end{tabular}
\end{center}

\QTP{Body Math}
Define an $IFS$ $A=(\mu _{A},\gamma _{A})$ of an AG-groupoid $S$ as follows: 
$\mu _{A}(1)=\mu _{A}(2)=\mu _{A}(3)=0.3,$ $\mu _{A}(4)=0.1$, $\mu
_{A}(5)=0.4,$ $\gamma _{A}(1)=0.2,$ $\gamma _{A}(2)=0.3,$ $\gamma
_{A}(3)=0.4,$ $\gamma _{A}(4)=0.5,$ $\gamma _{A}(5)=0.2.$ Now again define
an $IFS$ $B=(\mu _{B},\gamma _{B})$ of an AG-groupoid $S$ as follows: $\mu
_{B}(1)=\mu _{B}(2)=\mu _{B}(3)=0.5,$ $\mu _{B}(4)=0.4$, $\mu _{B}(5)=0.6,$ $%
\gamma _{B}(1)=0.3,$ $\gamma _{B}(2)=0.4,$ $\gamma _{B}(3)=0.5,$ $\gamma
_{B}(4)=0.6,$ $\gamma _{B}(5)=0.3.$ Then it is easy to observe that $A=(\mu
_{A},\gamma _{A})$ and $B=(\mu _{B},\gamma _{B})$ are an intuitionistic
fuzzy two-sided ideals of $S$ such that $(\mu _{A}\circ \mu _{B})(a)=\{0.1,$ 
$0.3,$ $0.4\}=(\mu _{A}\cap \mu _{B})(a)$ for all $a\in S$ and similarly $%
(\gamma _{A}\circ \gamma _{B})(a)=(\gamma _{A}\cap \gamma _{B})$ for all $%
a\in S$, that is, $A\circ B=A\cap B$ but $S$ is not an intra-regular because 
$3\in S$ is not an intra-regular.

An $IFS$ $A=(\mu _{A},\gamma _{A})$ of an AG-groupoid is said to be
idempotent if $\mu _{A}\circ \mu _{A}=\mu _{A}$ and $\gamma _{A}\circ \gamma
_{A}=\gamma _{A},$ that is, $A\circ A=A$ or $A^{2}=A.$

\begin{lemma}
\label{idem}Every intuitionistic fuzzy two-sided ideal $A=(\mu _{A},\gamma
_{A})$ of an intra-regular AG-groupoid $S$ is idempotent.
\end{lemma}

\begin{proof}
Let $S$ be an intra-regular AG-groupoid and let $A=(\mu _{A},\gamma _{A})$
be an intuitionistic fuzzy two-sided ideal of $S.$ Now for $a\in S$ there
exists $x,y\in S$ such that $a=(xa^{2})y.$ Now by using $(4)$ and $(2),$ we
have%
\begin{equation*}
a=(x(aa))y=(a(xa))(ey)=(ae)((xa)y).
\end{equation*}%
\begin{eqnarray*}
(\mu _{A}\circ \mu _{A})(a) &=&\dbigvee\limits_{a=(ae)((xa)y)}\{\mu
_{A}(ae)\wedge \mu _{A}((xa)y)\}\geq \mu _{A}(ae)\wedge \mu _{A}((xa)y) \\
&\geq &\mu _{A}(a)\wedge \mu _{A}(a)=\mu _{A}(a).
\end{eqnarray*}

This shows that $\mu _{A}\circ \mu _{A}\supseteq \mu _{A}$ and by using
Lemma \ref{as}, $\mu _{A}\circ \mu _{A}\subseteq \mu _{A}$, therefore $\mu
_{A}\circ \mu _{A}=\mu _{A}.$ Similarly we can prove that $\gamma _{A}\circ
\gamma _{A}=\gamma _{A},$ which implies that $A=(\mu _{A},\gamma _{A})$ is
idempotent.
\end{proof}

\begin{theorem}
The set of intuitionistic fuzzy two-sided ideals of an intra-regular
AG-groupoid $S$ forms a semilattice structure with identity $\delta $, where 
$\delta =(S_{\delta },\Theta _{\delta }).$
\end{theorem}

\begin{proof}
Let $\mathbb{I}_{\mu \gamma }$ be the set of intuitionistic fuzzy two-sided
ideals of an intra-regular AG-groupoid $S$ and let $A=(\mu _{A},\gamma _{A})$%
, $B=(\mu _{B},\gamma _{B})$ and $C=(\mu _{C},\gamma _{C})$ are any
intuitionistic fuzzy two-sided ideals of $\mathbb{I}_{\mu \gamma }.$ Clearly 
$\mathbb{I}_{\mu \gamma }$ is closed and by Lemma \ref{idem}, we have $%
A^{2}=A$. Now by using Lemma \ref{fgh}, we get $A\circ B=B\circ A$ and
therefore, we have%
\begin{equation*}
(A\circ B)\circ C=(B\circ A)\circ C=(C\circ A)\circ B=(A\circ C)\circ
B=(B\circ C)\circ A=A\circ (B\circ C).
\end{equation*}

It is easy to see from Corollary \ref{cor} that $\delta $ is an identity in $%
\mathbb{I}_{\mu \gamma }.$
\end{proof}

\end{document}